\documentclass[12pt]{amsart}
\usepackage{amsthm,amsmath,amsfonts,amssymb,latexsym,amscd,mathrsfs,hyperref,cleveref}
\usepackage{graphicx}
\usepackage{amsaddr}
\usepackage{cite}

\theoremstyle{plain}
\newtheorem{thm}{Theorem}
\newtheorem{lem}{Lemma}

\theoremstyle{definition}

\newtheorem*{exa*}{Example}

\crefname{thm}{theorem}{theorems}
\crefname{lem}{lemma}{lemmas}

\newcommand{\bs}{\boldsymbol}

\newcommand{\mc}{\mathcal}

\def \a{\alpha} \def \b{\beta}  \def \e{\varepsilon}    \def \s{\sigma} \def \t{\theta} 

\numberwithin{equation}{section}

\renewcommand{\labelenumi}{\setlength{\labelwidth}{\leftmargin}
   \addtolength{\labelwidth}{-\labelsep}
   \hbox to \labelwidth{\theenumi.\hfill}}

\begin{document}
\title{Small fractional parts of polynomials}
\author{Roger Baker}

 \begin{abstract}
Let $k \ge 6$. Using the recent result of Bourgain, Demeter, and Guth \cite{bdg} on the Vinogradov mean value, we obtain new bounds for small fracitonal parts of polynomials $\a_kn^k + \cdots + \a_1n$ and additive forms $\b_1n_1^k + \cdots + \b_sn_s^k$. Our results improve earlier theorems of Danicic (1957), Cook (1972), Baker (1982, 2000), Vaughan and Wooley (2000), and Wooley (2013).
 \end{abstract}
 

\maketitle

\section{Introduction}\label{sec1}

Let $J_{s,k}(N)$ be the Vinogradov mean value,
 \[J_{s,k}(N) : = \int_{[0,1)^k}\Bigg|\sum_{n=1}^N
 e(x_kn^k + \cdots + x_1n)\Bigg|^{2s} dx_1 \ldots dx_k.\]
Here $s$ and $k$ are natural numbers. Recently Wooley \cite{wool3} (for $k=3$) and Bourgain, Demeter, and Guth \cite{bdg} (for $k \ge 4)$ have established the main conjecture for $J_{s,k}(N)$, namely
 \begin{equation}\label{eq1.1}
J_{s,k}(N) \ll_{k,\e} N^{s+\e} + N^{2s-k(k+1)/2 + \e}. 
 \end{equation}
Here $\e$ is an arbitrary positive number. In the present note we combine \eqref{eq1.1} with techniques from two earlier publications \cite{rcb3, rcb4} to obtain new bounds of the form
 \[\min_{1\, \le\, n\, \le\, N} \|\a_kn^k + \cdots + 
 \a_1n\| \ll_{k,\e} N^{-\mu_k + \e}\quad (k = 8, 9, \ldots)\tag{i}\]
(with arbitrary real numbers $\a_1, \ldots, \a_k$, $\b_1, \ldots, \b_s$ here and below);
 \begin{gather*}
\min_{1\, \le\, n\, \le\, N} \|\a_kn^k + \a_1n\| \ll_{k,\e} N^{-\rho_k+\e} \quad (k = 6, 7, \ldots)\tag{ii}\\[2mm]
\min_{\substack{0 \le n_1, \ldots, n_s \le N\\
(n_1, \ldots, n_s)\ne \bs{0}}} \|\b_1n_1^k + \cdots + \b_s n_s^k\| \ll N^{-\s_{s,k}+\e} \quad (k=6, 7, \ldots, s \ge 1).\tag{iii}
 \end{gather*}
 
 \begin{thm}\label{thm1}
Let $k \ge 8$. Then (i) holds with $\mu_k = 1/2k(k-1)$.
 \end{thm}

 \begin{thm}\label{thm2} 
 \begin{enumerate}
\item[(a)] Let $k \ge 6$. Then (ii) holds with $\rho_k = 1/k(k-1)$.

\item[(b)] Let $k \ge 6$. For a certain positive absolute constant $B$, (ii) holds with $\rho_k = 1/k(2\log k + B\log\log k)$.
 \end{enumerate}
 \end{thm}

 \begin{thm}\label{thm3} 
 \begin{enumerate}
\item[(a)] Let $k \ge 6$, $1 \le s \le k(k-1)$. Then (iii) holds with $\s_{s,k} = s/k(k-1)$.
 \medskip

\item[(b)] Let 
 \[F(J, s, k) = \min \left(\frac sJ, \ \max_{J+1\le h \le s}
 \min\left(\frac{(2h-2)(s-k)+4k-4}{h(s-k)+4h-4}\, , \,
 \frac{s-h+J+1}J\right)\right)\]
 \end{enumerate}

\noindent Then (iii) holds for $k \ge 6$, $s > k(k-1)$ with
 \[\s_{s,k} = F(k(k-1), s, k).\]
In particular,
 \[\min_{\substack{
 0 \le n_1, \ldots, n_s \le N\\
 (n_1, \ldots, n_s)\ne \bs 0}} \|\b_1n_1^6 + \cdots 
 + \b_s n_s^6\| \ll  N^{-s/30+\e}(1 \le s \le 56).\]
 \end{thm}

We note here the existing results in each case. Let $K = 2^{k-1}$.

(i) This is known with $\mu_k = 1/K$ $(2 \le k \le 8)$ (Baker \cite{rcb}) and $\mu_k = 1/4k(k-2)$ for $k \ge 9$ (Wooley \cite{wool2}).
 \medskip

(ii) Only the special case $\a_1 = 0$ has been considered separately from (i). Here the result is known with $\rho_2 = 4/7$ (Zaharescu \cite{zah}); $\rho_k = 1/K$ $(3 \le k \le 6)$ (Danicic \cite{dan}), while there are the values $\rho_7 = 1/57.23$, $\rho_8 = 1/69.66$, $\rho_9 = 1/82.08$, $\rho_{10} = 1/94.62$, $\rho_{11} = 1/107.27, \ldots, \rho_{20} = 1/222.16$, given by Vaughan and Wooley \cite{vw}, which are better than the present method gives (in the monomial case) for $k \ge 11$. There is an absolute positive constant $C$ such that, for $k \ge 6$,
 \begin{equation}\label{eq1.2}
\min_{1 \le n \le N} \|\a n^k\| \ll_{k, \e} N^{-1/k(\log k + C\log\log k)} 
 \end{equation}
(Wooley \cite{wool}).
 \medskip

(iii) This is known with $\s_{s,k} = s/K$ for $k \ge 2$, $1 \le s \le K$ (Cook \cite{cook}), and
 \[\s_{s,k} = F(K, s, k) \quad (k \ge 4, s > K)\]
(Baker \cite{rcb4}). For $k = 2, 3$ and $s > K$, see Baker \cite{rcb, rcb4}; for example, $\s_{3,2} = 9/8$ and $\s_{5,3} = 5/4$.
 \medskip

We refer the reader to Heath-Brown \cite{hb}, Wooley \cite{wool}, and Vaughan and Wooley \cite{vw} for results of the kind: for irrational $\a$, we have
 \[\|\a n^k\| < n^{-\tau_k}\]
for infinitely many $k$. For example, one may take $\tau_k = 1/9.028 k$ for every $k$ \cite{wool}.

\section{Bounds for Weyl sums}\label{sec2}

We suppose throughout (as we may) that $\e$ is sufficiently small and $N$ is sufficiently large in terms of $k$, $\e$; we write $\eta = \e^2$.

 \begin{thm}\label{thm4}
Let $k \ge 3$ and $\e > 0$. Suppose that the Weyl sum
 \[g_k(\bs\a; N) : = \sum_{n=1}^N 
 e(\a_kn^k + \cdots + \a_1n)\]
satisfies
 \begin{equation}\label{eq2.1}
|g_k(\bs\a; N)| \ge A > N^{1 - 1/2k(k-1)+\e}. 
 \end{equation}
Then there exist integers $q$, $a_1, \ldots, a_k$ such that
 \begin{equation}\label{eq2.2}
1 \le q \le N^\e(NA^{-1})^k 
 \end{equation}
and
 \begin{equation}\label{eq2.3}
|q\, \a_j - a_j| \le N^{-j+\e}(NA^{-1})^k \quad (1 \le j \le k). 
 \end{equation}
If $\a_{k-1} = \cdots = \a_2 = 0$, then the same conclusion holds with the weaker lower bound.
 \begin{equation}\label{eq2.4}
|g_k(\bs \a; N)| \ge A > N^{1-1/k(k-1)+\e} 
 \end{equation}
in place of \eqref{eq2.1}.
 \end{thm}
 
 \begin{proof}
We initially proceed exactly as in the proof of \cite[Theorem 4.3]{rcb3} with $\t$ replaced by 0 and $\ell$ replaced by $(k-1)/2$. This is permissible since we have
 \[J_{s, k-1}(N) \ll N^{s+\e}\]
with $s = k(k-1)/2$, in place of the bound for $J_{s,k-1}(N)$ used in \cite{rcb3}. We find that for $j=2, \ldots, k$ there are coprime pairs of integers $q_j$, $b_j$ with
 \begin{gather*}
1 \le q_j \ll (NA^{-1})^{k(k-1)}(\log N)^C\\
|q\a_j - b_j| \le N^{-j+\e}(NA^{-1})^{k(k-1)}
 \end{gather*}
where we shall use $C$ for an unspecified positive constant depending on $k$. Let $q_0$ be the l.c.m of $q_2, \ldots, q_k$. We now follow the argument of \cite[pp. 41--42]{rcb3} to obtain
 \begin{equation}\label{eq2.5}
q_0 \ll (\log N)^C (NA^{-1})^{k(k-1)}. 
 \end{equation}

It follows that, with $a_j = q_0b_j/q_j$, we have
 \begin{equation}\label{eq2.6}
|q_0 \a_j - \a_j| \le N^{-j+2\e}(NA^{-1})^{2k(k-1)}\quad (j=2, \ldots, k).
 \end{equation}

We now appeal to Lemma 4.6 of \cite{rcb3}, which we restate here for clarity as \Cref{lem1}.

 \begin{lem}\label{lem1}
Suppose that there are integers $r$, $v_2, \ldots, v_k$ such that $\gcd$ $(r, v_2, \ldots, v_k) = 1$,
 \begin{equation}\label{eq2.7}
|q_j r - v_j| \le N^{1-j}/4k^4 \quad (j=2,\ldots, k), 
 \end{equation}
and that
 \begin{equation}\label{eq2.8}
|g_k(\bs\a; N)| \ge H > r^{1-1/k}N^\e. 
 \end{equation}
There is a natural number $t \le 2k^2$ such that
 \begin{gather}
tr \le (NH^{-1})^k N^\e,\label{eq2.9}\\[2mm]
t|\a_j r - v_j| \le (NH^{-1})^k N^{-j+\e}\quad (j=2, \ldots, k)\label{eq2.10}\\[2mm]
\|tr\, \a_1\| \le (NH^{-1})N^{-1+\e}.\label{eq2.11}
 \end{gather}
 \end{lem}

We now apply the lemma with $A = H$, $r = q_0d^{-1}$, $v_j = a_jd^{-1}$ where $d = \gcd(q_0, a_2, \ldots, a_k)$. From \eqref{eq2.5} and \eqref{eq2.6},
 \[|\a_j r - v_j| \le N^{-j+2\e}
 (NA^{-1})^{2k(k-1)} \le N^{-j+1}(4k^4)^{-1}\]
since
 \[(NA^{-1})^{2k(k-1)} \le N^{1-12\e}\]
and $r \le N^{1-5\e}$,
 \[A r^{-1+1/k} N^{-2\e} \ge 
 N^{1-1/k(k-1)-1+1/k-C\e} \gg 1.\]
The inequalities \eqref{eq2.9}--\eqref{eq2.11} now yield the first assertion of the theorem with $q =tr$. For the second assertion, since $\a_2, \ldots, \a_{k-1}$ are 0, we may take $r=q_k$, $v_k = b_k$, $v_2 = \cdots = v_{k-1} = 0$, $H = A$ in the application of \Cref{lem1}. (The inequality \eqref{eq2.4} suffices in the earlier part of the argument.) We know that
 \[|r\a_k - a_k| \le N^{-k+\e}(NA^{-1})^{k(k-1)}\]
rather than the weaker bound \eqref{eq2.6}. We may now complete the proof in the same way as before.
 \end{proof}
 
\section{Proof of Theorems \ref{thm1}, \ref{thm2}, and \ref{thm3}}\label{sec3}

 \begin{proof}[Proof of \Cref{thm1}]
Suppose there is no solution of
 \begin{equation}\label{eq3.1}
1 \le n \le N, \ \|\a_kn^k + \cdots + \a_1n\| \le N^{-1/J+\e} 
 \end{equation}
where $J$ denotes $2k(k-1)$. By \cite[Theorem 2.2]{rcb3} we have
 \[\sum_{m=1}^M |g_k(m\bs \a; N)| > N/6,\]
where $M = [N^{+1/J-\e}]$. There is an integer $m$, $1 \le m \le M$ such that
 \[|g_k(m\a; N)| > A = N/6M.\]
We have
 \[(NA^{-1})^{2k(k-1)} \ll M^{2k(k-1)}
 \ll N^{1 - 2k(k-1)\e}.\]
By \Cref{thm4} there is a natural number $q = tr$ such that
 \begin{align}
q \ll N^\e &(NA^{-1})^k \ll M^k,\label{eq3.2}\\[2mm]
\|q m \a_j\| &\ll (NA^{-1})^k N^{-j+\e}\label{eq3.3}\\[2mm]
&\ll M^k N^{-j+\e}\quad (j = 1, \ldots, k).\notag
 \end{align}
Now let $n = qm$. Then
 \begin{align*}
n \ll M^{k+1} &\ll N^{(k+1)/J} \ll N^{1-\e},\\[2mm]
\|n^j\a_j\| &\le n^{j-1} \|n\a_j\|\\[2mm]
&\ll M^{(k+1)(j-1)+k} N^{-j+\e} \ll M^{-1}N^{-\e}
 \end{align*}
since $M^{(k+1)j} \ll N^{(k+1)j/J-(k+1)\e} \ll N^{j-2\e}$. It follows that $n$ satisfies \eqref{eq3.1}, which is a contradiction. This completes the proof of \Cref{thm1}.
 \end{proof}
 
 \begin{proof}[Proof of \Cref{thm2}(a)]
We follow the above proof; this time, $J$ denotes $k(k-1)$. The second assertion of \Cref{thm4} provides an integer $q = tr$ satisfying \eqref{eq3.2}, and \eqref{eq3.3} for the relevant values $j=1$, $k$. Now we complete the proof as before.
 \bigskip

 \begin{proof}[Proof of \Cref{thm2}(b)]
This is a simple consequence of Wooley's bound \eqref{eq1.2}. Let $\nu = \nu(k)$ have the property that
 \[\min_{1\, \le\, n\, \le\, N} \|\a n^k\|
 \ll_k N^{-\nu}\]
for $N \ge 1$ and real $\a$. Let $a = \frac 1{2+\nu}$, $b = 1 - a$. By Dirichlet's theorem there is a natural number $\ell \le N^b$ with
 \[\|\a_1 \ell\| \le N^{-b}.\]
We now choose another natural number $m \le N^a$ with
 \[\|\a_k \ell^k m^k\| \ll N^{-a\nu} = 
 N^{-\nu/(2+\nu)}.\]

Note that
 \[\|\a_1 \ell m\| \le N^{a-b} = N^{2a-1}.\]
Since $2a - 1 = -\frac \nu{2+\nu}$, we have, with $n = \ell m$,
 \[1 \le n \le N, \ \|\a_k n^k + \a_1n\|
 \ll N^{-\nu/(2+\nu)}.\]
Taking $\nu = 1/k(\log k + C \log\log k)$, we obtain
 \[\frac \nu{2 + \nu} = \frac 1{2k\log k + 2C\log\log k+1},\]
so that \Cref{thm2}(b) holds with a suitable choice of $B$.
 \end{proof}
 
 \begin{exa*}
If we take $k = 20$, $\nu = 1/222.16$ from \cite{vw}, we obtain the value $1/445.32$ for $\rho_{20}$, which is not as good as \Cref{thm2}(a). The proof of \Cref{thm2}(b) is relatively crude, so it may be possible to do better using ideas from \cite{vw}, \cite{wool}.
 \end{exa*}
 
\noindent\textit{Proof of \Cref{thm3}(b).}
We can follow the proof of Theorem 1.8 of \cite{rcb4} (in the case $k \ge 4$) verbatim, replacing $K$ by $J:=k(k-1)$. The role of Lemma 5.2 of \cite{rcb4} is played by \Cref{thm4} in conjunction with \cite[Lemma 8.6]{rcb3}. 
\bigskip

\noindent\textit{Proof of \Cref{thm3}(a).}
Write $J=k(k-1)$ again. We assume that there is no solution of
 \begin{equation}\label{eq3.4}
\|\b_1 n_1^k + \cdots + \b_s n_s^k\| \le N^{-s/J + \e} 
 \end{equation}
with $0 \le n_1, \ldots, n_s \le N$, $(n_1, \ldots, n_s) \ne \bs 0$. Let 
 \[S_i(m) = \sum_{n=1}^N e(m\b_in^k), \quad
 L = [N^{1/J-\e}].\]
Following \cite{rcb4}, Lemma 5.1, we find that there is a set $\mc B$ of natural numbers, $\mc B \subset [1, L]$, and there are positive numbers $B_1 \ge \cdots \ge B_s$ such that
 \[B_i < |S_i(m)| \le 2B_i \quad
 (i=1, \ldots, s)\]
and
 \[B_1 \ldots B_s\, |\mc B| \gg N^{s-\eta}.\]
(This may require a reordering of $\b_1, \ldots, \b_s$.) We can now follow the proof of Lemma 5.4 on \cite{rcb4}, with $K$ replaced by $J$, to obtain the inequality
 \[|\mc B| \ll LN^{-1+2k\eta} |\mc B|^{k/s}.\]

Suppose first that $s > k$. Then
 \[LN^{-1 + 2k\eta} \gg |\mc B|^{1-k/s} \gg 1,\]
contrary to the definition of $L$.

Suppose now that $s \le k$. Then
 \begin{align*}
L^{\frac ks - 1} &\ge |\mc B|^{\frac ks - 1} \gg L^{-1} N^{1-2k\eta},\\[2mm]
L &\gg N^{\frac sk - 2s\eta}. 
 \end{align*}
This is again contrary to the definition of $L$, and we conclude that there is a solution of \eqref{eq3.4}.
 \end{proof}

 \end{document}